\documentclass[reqno,12pt]{amsart}
\newcommand{\RR}{{\mathbb R}}
\newcommand{\CC}{{\mathbb C}}
\newcommand{\ZZ}{{\mathbb Z}}

\newcommand{\NN}{{\mathbb N}}

\newcommand{\mcS}{\mathcal{S}}

\def\bege{\begin{equation}} \def\ende{\end{equation}}
\def\begr{\begin{eqnarray}} \def\endr{\end{eqnarray}}
\usepackage{amsmath}
\usepackage{amssymb}
\usepackage{cite}

\usepackage{bbm}
\usepackage{amsmath}
\usepackage{txfonts}
\usepackage{stmaryrd}
\usepackage{amssymb}
\usepackage{amsfonts}
\usepackage{times}
\usepackage{mathrsfs}

\usepackage{amsthm}
\usepackage{enumerate}

\usepackage{framed}
\usepackage{lipsum}
\usepackage{color}

\newtheorem{theorem}{\hspace{2em}Theorem}

\newtheorem{lemma}{\hspace{2em}Lemma}

\newtheorem*{thA}{\hspace{2em}Theorem A}
\newtheorem*{thB}{\hspace{2em}Theorem B}

\newcommand{\TT}{{\mathbb T}}

\def\CC{ \mathbb{C}}
\newcommand{\DD}{{\mathbb D}}
\def\B{\mathcal{B}}
\def\R{\mathcal{R}}

\def\T{\mathcal{T}}

\def\D{\mathbb{D}}
\def\N{\mathbb N}

\def\hD{\hat{\mathcal{D}}}

\def\om{\omega}

\def\p{{\prime}}

\def\begr{\begin{eqnarray}} \def\endr{\end{eqnarray}}

\def\msk{\medskip}
\def\ol{\overline}
\newtheorem{Corollary}{Corollary}

\begin{document}
\title[]{The generalized Volterra  integral operator  and  Toeplitz operator  on weighted Bergman spaces}
 \author{ Juntao Du,    Songxiao Li$\dagger$ and Dan Qu}

 \address{Juntao Du\\ Department of mathematics, Shantou University, Shantou, Guangdong,  515063, China.}
 \email{jtdu007@163.com  }

 \address{Songxiao Li\\ Institute of Fundamental and Frontier Sciences, University of Electronic Science and Technology of China,
 610054, Chengdu, Sichuan, P.R. China. }
 \email{jyulsx@163.com}

\address{Dan Qu \\  Faculty of Information Technology, Macau University of Science and Technology, Avenida Wai Long, Macau, China }
\email{qd4027@163.com}

 \begin{abstract}
 We study the boundedness and compactness of the generalized Volterra integral operator on weighted Bergman spaces with doubling weights  on the unit disk. A generalized   Toeplitz operator is defined and  the boundedness, compactness and Schatten class of this operator are investigated  on the Hilbert weighted Bergman space. As an application,    Schatten class membership of   generalized Volterra integral operators  are also characterized. Finally, we also get the characterizations of Schatten class membership of generalized   Toeplitz operator  and generalized Volterra integral operators on the Hardy space $H^2$.
 \thanks{$\dagger$ Corresponding author.}
 \end{abstract}
 \maketitle

\section{Introduction}

Let $\D$ be the open unit disk of  complex plane $\CC$ and $H(\D)$ be the space of all analytic functions on $\D$. A function  $\om: \D \rightarrow [0,\infty)$ is called a weight if it is  positive and integrable. $\om$ is radial if $\om(z)=\om(|z|)$ for all $z\in\DD$.

 Let  $\om$ be a  radial weight. For $r\in [0,1)$, set $$\hat{\om}(r)=\int_r^1 \om(s)ds.$$
We say that $\om$ is a doubling weight, denoted by $\om\in \hD$, if  there is a constant $C>0$ such that
$\hat{\om}(r)<C\hat{\om}(\frac{1+r}{2})$.
 $\om$ is called a regular weight, denoted by $\om \in \R$,  if    there is a constant $C>0$ and $0<\delta<1$ such that
$$\frac{1}{C}<\frac{\hat{\om}(r)}{(1-r)\om(r)}<C, \,\,\mbox{ whenever }\,\,\delta< r<1.$$
  From \cite{PjaRj2014book}, we see that $\R \subset \hD $ and  $\hat{\om}\in\R$ if $\om\in \hD$. Moreover, for $\om\in\hD$ and any given $\varepsilon \in (0,1)$,
$$\om^*(z)\approx (1-|z|)\hat{\om}(z),~~~ \varepsilon<|z|<1.$$
  Here
$$\om^*(z)=\int_{|z|}^1 s\om(s)\log \frac{s}{|z|}ds,\,\,\,z\in\D\backslash\{0\}.$$
See \cite{Pja2015,PjaRj2014book} for more details about $\R$  and $\hD$.

 Let $d\sigma$ denote the normalized    length  measures on   $\TT$, the boundary of $\D$. For $0<p<\infty$, the  Hardy space $H^p$  consists of all functions $f\in H(\D)$ such that
 $$\|f\|^p_{H^p}=\sup_{0<r<1} \int_{\TT}|f(r\xi)|^pd\sigma(\xi) <\infty.$$

%Suppose $\mu$ is  a positive Borel measure on $\D$ and $0<p<\infty$.
%The Lebesgue spaces $L^p(\D,d\mu)$ (or $L_\mu^p$ for brief) and $L^p(\TT, \sigma)$ (or $L^p(\TT)$ for brief) are defined in a standard way.

Let $0<p<\infty$ and $\om\in\hD$. The weighted Bergman space $A_\om^p$, induced by doubling weight $\om$,  consists of all functions $f \in H(\D)$ such that
$$\|f\|_{A_\om^p}=\left(\int_{\D}|f(z)|^p\om(z)dA(z)\right)^\frac{1}{p}<\infty,
  $$
 where $dA $ is the normalized Lebesgue area measure on $\D$.
When $\om(z)=(1-|z|^2)^\alpha(\alpha>-1)$, the space $A_\om^p$ becomes the classical weighted Bergman space $A_\alpha^p$.
When $\alpha=0$, we will write $A_0^p$ as $A^p$.

 The Bloch space, denoted by $\B$,   is the space of all $f\in H(\DD)$ such that
$\sup_{z\in\DD}(1-|z|^2)|f'(z)|<\infty.$
The little Bloch space $\B_0,$ consists of all $f\in\B$ such that
$\lim_{|z|\to 1}(1-|z|^2)|f'(z)|=0.$
For $\om\in\hD$, let   $\mathcal{C}^1(\om^*)$ denote the space consisting of all $g\in H(\D)$ such that
\begin{align*}%\label{0703-1}
\sup_{a\in\D}\frac{\int_{S_a}|g^\p(z)|^2\om^*(z)dA(z)}{\om(S_a)}<\infty.
\end{align*}
Here
$S_a=\left\{re^{\mathrm{i\theta}}:|a|<r<1, |\arg a-\theta|<\frac{1-|a|}{2}\right\}$
is  a Carleson square at $a\in\D$. We say that $g\in\mathcal{C}_0^1(\om^*)$ if   $g\in\mathcal{C}^1(\om^*)$ and
$$\lim_{|a|\to 1}\frac{\int_{S_a}|g^\p(z)|^2\om^*(z)dA(z)}{\om(S_a)}=0.$$
From \cite[Proposition 5.1]{PjaRj2014book}, we see that $\mathcal{C}^1(\om^*)\subseteq \B$  and  $C_0^1(\om^*)\subseteq \B_0$ when $\om\in\hD$. In particularly, $C^1(\om^*)=\B$ and $C_0^1(\om^*)=\B_0$ when $\om\in \R$.

Suppose  $g\in H(\D)$. The integral operator $T_g$,
called the Volterra   integral operator, is defined  by
\begr
T_gf(z)=\int_0^z f(\xi)g'(\xi)d\xi , ~ \qquad~~ ~f\in H(\D), ~~  z\in \D. \nonumber
\endr
 The operator $T_g$ was first introduced  by Pommerenke in \cite{Pc1977cmh}.  He showed that $T_g$  is   bounded on the Hardy space  $H^2$ if and only if $g \in BMOA$, the bounded mean oscillation analytic function space. % Then, the Volterra operators attract a lot of attentions.
In \cite{AaSa1997iumj}, Aleman and Siskakis proved that $T_g$ is bounded on  Bergman space $A^p$  if and only if   $g\in\B$.
In \cite{AaCj2001jam}, Aleman and Cima completely characterized the boundedness of Volterra operator on Hardy spaces.
 In particular, they showed that $T_g$ is bounded on $H^p$ if and only if $g\in BMOA$.
More results about the Volterra operator on some analytic function spaces can be seen in \cite{AaSa1995cv,Hz2004jmaa,HzLj2018jga,ls1, ls2,MsPjPaWm2019jfa,Pj2016jfa,PjaRj2014book,zhu} and the referees therein.
%In 2004, Hu \cite{Hz2004jmaa} extended the Volterra operator to the unit ball of $\CC^n$.
%Let $X$ and $Y$ be any given classical weighted Bergman space or Hardy space on the unit disk( or the unit ball), the boundedness of $T_g:X\to Y$ was completely characterized by different authors, for example, see  \cite{ls1, ls2,MsPjPaWm2019jfa,Pj2016jfa,PjaRj2014book,zhu} and the referees therein.

In \cite{PjaRj2014book}, Pel\'aez and R\"atty\"a characterized the boundedness, compactness and Schatten class of $T_g$ between Bergman spaces induced by rapidly increasing weights.
Using Theorem 5 in \cite{KtRj2019mz}, those results can be  extended  to $\om\in \hD$ as follows.
%We display some of them as follows.

\begin{thA}
Let $0<p,q<\infty, \om\in\hD$ and $g\in H(\D)$. 
\begin{enumerate}[(i)]
  \item If $0<p<q<\infty$,  then $T_g:A_\om^p\to A_\om^q$ is bounded (compact) if and only if $$\sup\limits_{\frac{1}{2}<|z|<1}\frac{(1-|z|)|g^\p(z)|}{(\om^*(z))^{\frac{1}{p}-\frac{1}{q}}}<\infty
      \left(\lim\limits_{|z|\to 1}\frac{(1-|z|)|g^\p(z)|}{(\om^*(z))^{\frac{1}{p}-\frac{1}{q}}}=0\right)
      .$$
   \item $T_g:A_\om^p\to A_\om^p$ is bounded (compact) if and only if $g\in\mathcal{C}^1(\om^*)$ ($g\in\mathcal{C}_0^1(\om^*)$).
  \item If $0<q<p<\infty$, then the following conditions are equivalent:
  \begin{enumerate}
    \item $T_g:A_\om^p\to A_\om^q$ is bounded;
    \item $T_g:A_\om^p\to A_\om^q$ is compact;
    \item $g\in A_\om^s$, where $\frac{1}{s}=\frac{1}{q}-\frac{1}{p}$.
  \end{enumerate}
\end{enumerate}
\end{thA}

In 2020,  Chalmouks \cite{Cn2020pams} introduced a natural generalization of the Volterra operator,
that is, for any given $g\in H(\D)$ and $k,n\in\ZZ$ such that $0\leq k<n<\infty$,
$$T_{g}^{n,k} f(z)= I^n(f^{(k)}g^{(n-k)})(z), \,\,\,\,\forall f\in H(\D) \mbox{ and } z\in\D.$$
Here,  $I^n$ is the $n$-th iteration of the integration operator $I f(z)=\int_0^z f(t)dt$.
Obviously, when $n=1$ and $k=0$, $T_g^{1,0}=T_g$.
%that is, for a symbol $g\in H(\D)$ and a  sequence of complex number $a=\{a_k\}_{k=1}^{n-1}$, let
%$$T_{g,a} f(z)=I^n (fg^{(n)}+a_1 f^\p g^{(n-1)}+\cdots+a_{n-1}f^{(n-1)}g^\p)(z),\,\,f\in H(\D).$$
%Here,
%For convenience, let
%$$T_g^{n,k}f(z)=I^n(f^{(k)}g^{(n-k)})(z), \,\,\,k=0,1,\cdots,n-1,$$
%and then,
%$$T_{g,a} f=T^{n,0}_g f+\sum_{k=1}^{n-1} a_k T_g^{n,k}f,\,\,f\in H(\D).$$
%
%
%In \cite{Cn2020pams}, the author investigate the  boundedness of $T_{g,a}:H^p\to H^q$.
%When $0<p\leq q<\infty$,  the boundedness of $T_{g,a}:H^p\to H^q$ is only  determined by  that of
%$T_{g}^{n,0}$ and $T_{g}^{n,k}:H^p\to H^q(k=1,\cdots,n-1)$ such that $a_k\neq 0$.
The following result can be deduced from \cite{Cn2020pams} directly.
\begin{thB}
Let $0<p,q<\infty, k,n\in\mathbb{Z}$ satisfying $0\leq k<n$, and $g\in H(\D)$. 
\begin{enumerate}[(i)]
  \item If $0<p<q<\infty$,   $T_g^{n,k}:H^p\to H^q$ is bounded  if and only if
      $$\sup\limits_{z\in\D}(1-|z|^2)^{\frac{1}{q}-\frac{1}{p}+n-k}|g^{(n-k)}(z)|<\infty.$$
  \item  $T_g^{n,0}: H^p\to H^p$ is bounded if and only if $g\in BMOA$ .
  \item If $k\geq 1$, $T_g^{n,k}: H^p\to H^p$ is bounded  if and only if $g\in\B$.
  \item If  $0<q<p<\infty$ and  $\frac{1}{s}=\frac{1}{q}-\frac{1}{p}$,  $g\in H^s$ implies the boundedness of $T_g^{n,k}:H^p\to H^q$;
  the boundedness of $T_g^{2,0}:H^p\to H^q$ implies that $g\in H^s$.
\end{enumerate}
\end{thB}

In this paper, motivated by  \cite{Cn2020pams,PjaRj2014book}, we study the boundedness, compactness and Schatten class membership of  the generalized Volterra operator $T_g^{n,k}$ between $A_\om^p$ and $A_\om^q$ when $\om\in\hD$ and $0<p,q<\infty$. Moreover, we define a new operator, i.e., the generalized   Toeplitz operator. We study the boundedness, compactness and Schatten class membership of  the generalized  Toeplitz operator on $A_\om^2$ and  on the Hardy space $H^2$.

The paper is organized as follows. In Section 2, we state  some preliminary results. In Section 3, the boundedness and compactness of $T_g^{n,k}:A_\om^p\to A_\om^q$ are investigated when either $0<p\leq q<\infty$ or $2\leq q<p<\infty$ and $k=0$. In Section 4, the boundedness, compactness and Schatten $p$-class of generalized Toeplitz operators on $A_\om^2$ are studied. As an application, we obtain the characterization of  Schatten class generalized Volterra operator. In the last section, we  also investigate  the  Schatten class generalized   Toeplitz operator  and generalized Volterra integral operator  on the Hardy space $H^2$.

Throughout this paper, the letter $C$ will denote  constants and may differ from one occurrence to the other.
The notation $A \lesssim B$ means that there is a positive constant C such that $A\leq CB$.
The notation $A \approx B$ means $A\lesssim B$ and $B\lesssim A$.\msk

\section{Preliminaries}

In this section, we state some lemmas for the proof of our main results. First, we introduce some notations.
Let
$$\Gamma_z=\left\{re^{\mathtt{i}\theta}\in\D: |\theta-\arg z|<\frac{1}{2}\left(1-\frac{r}{|z|}\right)\right\},\,\,z\in\ol{\D}\backslash\{0\}$$
and
$$T_u=\{z\in\D:u\in\Gamma_z\}, \,\,\,u\in\D.$$
By a calculation, we have
\begin{align}\label{0823-1}
T_u\subset S_u  ~~\mbox{ and }~~  \om(T_u)\approx \om(S_u)\approx (1-|u|)\hat{\om}(u).
\end{align}

\begin{lemma}[\cite{PjaRj2014book}]\label{0413-2}
Suppose $0<p<\infty$, $n\in\N$ and $f\in H(\D)$. Let $\om$ be a radial weight. Then
\begin{align*}
\|f\|_{A_\om^p}^p &=\om(\D)|f(0)|^p  + p^2 \int_\D |f(z)|^{p-2}|f^\p(z)|^2\om^*(z)dA(z)   \\
&\approx \sum_{k=0}^{n-1}|f^{(k)}(0)|^p+\int_\D \left(\int_{\Gamma_z}|f^{(n)}(u)|^2\left(1-\frac{|u|}{|z|}\right)^{2n-2}dA(u)\right)^\frac{p}{2}\om(z)dA(z)\\
&\approx \int_\D |(Nf)(z)|^p\om(z)dA(z).
\end{align*}
Here,  $(Nf)(z)=\sup\limits_{u\in\Gamma_z}|f(u)|$.
\end{lemma}

Let $f\in H(\D)$ and $k\in\N$. The Paley-Littlewood $G_k$-function of order $k$  is defined by
$$G_k(f)(e^{i\theta})=\left(\int_0^1 |f^{(k)}(re^{i\theta})|^2(1-r)^{2k-1}dr\right)^\frac{1}{2}.$$
In \cite{Cn2020pams}, Chalmoukis proved the following result, which gave an equivalent norm on $H^p$.

\begin{lemma}\label{0417-1}
Let $k\in\N$,  $0<p<\infty$ and $f\in H(\D)$ with $f^{(i)}(0)=0$, $0\leq i <k$. Then,
$$ \|f\|_{H^p}^p\approx \int_0^{2\pi} \left(G_k(f)(e^{i\theta})\right)^pd\theta.$$
\end{lemma}

In \cite{PjaRjSk2015mz}, Pel\'aez and R\"atty\"a estimated the norm of identity operator $I_d:A_\om^p\to L_\mu^q$.
In \cite{PjaRj2015}, Pel\'aez and R\"atty\"a completely described the boundedness and compactness of    $D^k: A_\om^p\to L_\mu^q$.
By the proof of Lemma 8 and Proposition 2 in \cite{PjaRj2015}, it is easy to estimate $\|D^k\|_{A_\om^p\to L_\mu^q}$ when $k\in\N\cup \{0\}$ and $0<p<q<\infty$.
As usual, $D^0=I_d$.
Now, we can state some results in \cite{PjaRj2015,PjaRjSk2015mz} as follows.

\begin{lemma} \label{0516-1}
Let $0<p\leq q<\infty$, $k\in \N\cup\{0\}$, $\om\in\hD$ and $\mu$ be a positive Borel on $\D$.
\begin{enumerate}[(i)]
  \item If $0<p<q<\infty$, then, for any fixed $r\in(0,\infty)$,
      $$\|D^k\|_{A_\om^p\to L_\mu^q}^q \approx\sup_{z\in\D} \frac{\mu(\D(z,r))}{(1-|z|)^{qk}\om(S_z)^\frac{q}{p}}<\infty.$$
Here  $\D(z,r)$ is the Bergman disk.
  \item If $0<q<p<\infty$, then
  $$\|I_d\|_{A_\om^p\to L_\mu^q}^q\approx
  \left(\int_\D \left(\int_{\Gamma_z}\frac{d\mu(\xi)}{\om(T_\xi)}\right)^\frac{p}{p-q}\om(z)dA(z)\right)^\frac{p-q}{p}.$$
      %\left(\lim_{|z|\to 1-} \frac{\mu(D(z,r))}{(1-|z|)^{qk+\frac{q}{p}}\hat{\om}(z)^\frac{q}{p}}=0\right)
 % \item when $k\in\N$ and $0<p=q<2$, $D^k:A_\om^p\to L_\mu^q$ is bounded (compact) if and only if, for any given $r\in(0,\infty)$,
  %$$afdf$$
\end{enumerate}
\end{lemma}

\begin{lemma}[\cite{PjaRj2014book}]\label{0425-1}
Suppose $\om\in\hD$ and $0<p<\infty$. If $\gamma$ is large enough, then the function $F_{a}(z)=\left(\frac{1-|a|}{1-\ol{a}z}\right)^\gamma$ belongs to $ A_\om^p$ and $\|F_a\|_{A_\om^p}^p\approx (1-|a|)\hat{\om}(a)$.
\end{lemma}

The following lemma is always used to investigate the compactness of linear operators between Bergman spaces induced by doubling weights. See \cite[Lemma 3.3]{CmPjPcRj2016jfa} for example.
\begin{lemma}\label{0425-2}
Suppose $\om\in\hD$ and $0<p,q<\infty$. If $T:A_\om^p\to A_\om^q$ is bounded and linear, $T$ is compact
if and only if $\lim\limits_{j\to\infty}\|Tf_j\|_{A_\om^q}=0$ whenever $\{f_j\}$ is bounded in $A_\om^p$ and converges to 0 uniformly on any given compact subset of $\D$.
\end{lemma}

Theorem 1.34 in \cite{zhu} plays a very important role when we investigate the Schatten class operator on a Hilbert space.

\begin{lemma}\label{0517-1}
Suppose $T$ is compact on a Hilbert space $H$ and $\{\lambda_j\}$ is the decreasing singular value sequence of $T$.
\begin{enumerate}[(i)]
  \item For each $j\geq 1$, we have
  $$\lambda_{j+1}=\inf\left\{\|T-F\|_{H\to H}:F\in\mathcal{F}_j(H)\right\},$$
  where $\mathcal{F}_j(H)$ is the set of all linear operators on $H$ with rank less than or equal to $j$.
  \item For each $j\geq 0$, we have
  $$\lambda_{j+1}=\min_{f_1,f_2,\cdots,f_j\in H}\max\left\{ \|Tf\|_{H}:\|f\|_{H}=1, \langle f,f_i\rangle_{H}=0,i=1,2,\cdots,j \right\}.$$
\end{enumerate}
\end{lemma}

\section{Generalized Integral operators}
\begin{theorem}\label{0514-11}
Suppose $0<p<q<\infty$, $\om\in\hD$, $k,n\in\mathbb{Z}$ satisfying  $0\leq k<n$, and $g\in H(\D)$.
\begin{enumerate}[(i)]
  \item
   $T_g^{n,k}:A_\om^p\to A_\om^q$ is bounded if and only if
   \begin{align}\label{0413-1}
   \sup_{a\in\D} \frac{(1-|a|)^{n-k}|g^{(n-k)}(a)|}{(1-|a|)^{\frac{1}{p}-\frac{1}{q}}\hat{\om}(a)^{\frac{1}{p}-\frac{1}{q}}}<\infty.
   \end{align}
  \item
   $T_g^{n,k}:A_\om^p\to A_\om^q$ is compact if and only if
   \begin{align}\label{0514-3}
   \lim_{|a|\to 1} \frac{(1-|a|)^{n-k}|g^{(n-k)}(a)|}{(1-|a|)^{\frac{1}{p}-\frac{1}{q}}\hat{\om}(a)^{\frac{1}{p}-\frac{1}{q}}}=0.
   \end{align}
\end{enumerate}
\end{theorem}
\begin{proof}
{\it (i).}
Suppose that $T_g^{n,k}:A_\om^p\to A_\om^q$ is bounded. For any $f\in H(\D)$ and $r\in(0,1)$, let
$$M_q(r,f)=\left( \int_0^{2\pi} |f(re^{\mathtt{i}\theta})|^qd\theta  \right)^\frac{1}{q}, \,\,\,\,
 M_\infty(r,f)=\sup_{\theta\in [0,2\pi)} |f(re^{\mathtt{i}\theta})| .$$
From the well known facts that
\begin{equation*}
M_\infty(r,f)\lesssim \frac{M_{q}(\frac{1+r}{2},f)}{(1-r)^{\frac{1}{q}}},
\,\,\,
M_q(r,f^\p)\lesssim  \frac{M_{q}(\frac{1+r}{2},f)}{1-r},
\end{equation*}
we have
$$M^q_\infty(r,f^{(n)})\lesssim  \frac{\|f\|^q_{A_\om^q}}{(1-r)^{nq+1}\hat{\om}(r)}.$$
Suppose $\gamma$ is large enough. For any $a\in\D$,  let $F_a(z)=\left(\frac{1-|a|^2}{1-\ol{a}z}\right)^\gamma$.
By Lemma \ref{0425-1}, we obtain
\begin{align}\label{0417-2}
\frac{|g^{(n-k)}(a)|}{(1-|a|)^{k}}
\approx  |(T_{g}^{n,k}F_a)^{(n)}(a)|
\lesssim \frac{\|T_g^{n,k} F_a\|_{A_\om^q}}{(1-|a|)^{n+\frac{1}{q}}\hat{\om}(a)^\frac{1}{q}}
\lesssim \frac{\|T_g^{n,k}\|_{A_\om^p\to A_\om^q} }
{(1-|a|)^{n+\frac{1}{q}-\frac{1}{p}}\hat{\om}(a)^{\frac{1}{q}-\frac{1}{p}}},
\end{align}
which deduces (\ref{0413-1}).

Suppose that (\ref{0413-1}) holds. We will divide the proof   into three cases.

{\it Case 1: $q=2$.} By Lemma \ref{0413-2}, Fubini's Theorem and (\ref{0823-1}), we have
\begin{align}
\|T_g^{n,k}f\|_{A_\om^2}^2
&\approx \int_\D \left(\int_{\Gamma_z}|f^{(k)}(u)g^{(n-k)}(u)|^2\left(1-\frac{|u|}{|z|}\right)^{2n-2}dA(u)\right)\om(z)dA(z)  \nonumber\\
&\lesssim \int_\D \left(\int_{T_u}|f^{(k)}(u)g^{(n-k)}(u)|^2\left(1-|u|\right)^{2n-2}\om(z)dA(z)\right)dA(u) \nonumber\\
&\approx\int_\D |f^{(k)}(u)|^2|g^{(n-k)}(u)|^2(1-|u|)^{2n-2}\om(S_u)dA(u). \label{0514-5}
\end{align}
Here and henceforth,  for  any given $\beta\in\RR$,
\begin{align}\label{0424-1}
d\mu_{\beta,k}(u)=|g^{(n-k)}(u)|^2(1-|u|)^\beta\om(S_u)dA(u).
\end{align}
Then, for any fixed $0<r<\infty$ and any $z\in \D$, by (\ref{0823-1}) and $\hat{\om}\in\R$, we have
\begin{align*}
\mu_{2n-2,k}(\D(z,r))
&=\int_{\D(z,r)}  |g^{(n-k)}(u)|^2(1-|u|)^{2n-2}\om(S_u)dA(u)\\
&\leq \sup_{u\in \D(z,r)} |g^{(n-k)}(u)|^2 \int_{\D(z,r)}  (1-|u|)^{2n-2}\om(S_u)dA(u)  \\
&\leq  \sup_{u\in \D(z,r)} \left(|g^{(n-k)}(u)|^2  (1-|u|)^{2n}\om(S_u)\right)\\
&\approx (1-|z|)^{2k}\om(S_z)^\frac{2}{p}\sup_{u\in \D(z,r)}\frac{|g^{(n-k)}(u)|^2  (1-|u|)^{2n}\om(S_u)}{(1-|u|)^{2k}\om(S_u)^\frac{2}{p}}.
\end{align*}
Thus, by  (\ref{0413-1}),
\begin{align}\label{0514-4}
\frac{\mu_{2n-2,k}(D(z,r))}{(1-|z|)^{2k}\om(S_z)^\frac{2}{p}}
\lesssim \sup_{u\in\D} \frac{|g^{(n-k)}(u)|^2(1-|u|)^{2n}\om(S_u)}{(1-|u|)^{2k}\om(S_u)^\frac{2}{p}}<\infty.
\end{align}
Then Lemma \ref{0516-1} deduces that $D^{k}:A_\om^p\to L_{\mu_{2n-2,k}}^2$ is bounded. So,  $\|T_g^{n,k}f\|_{A_\om^2}\lesssim \|f\|_{A_\om^p}$.

{\it Case 2: $2<q<\infty.$}  For any $f\in A_\om^p$ and $\|h\|_{L_\om^\frac{q}{q-2}}\leq 1$, let
\begin{align}\label{0704-2}
(M_\om h )(z)=\sup_{z\in S_a}\frac{\int_{S_a}|h(\eta)|\om(\eta)dA(\eta)}{\om(S_a)},
\end{align}
\begin{align}\label{0704-3}
J_{f,h}=\int_\D |h(z)|\left(\int_{\Gamma_z}|f^{(k)}(u)|^2|g^{(n-k)}(u)|^2\left(1-\frac{|u|}{|z|}\right)^{2n-2}dA(u)\right)\om(z)dA(z),
\end{align}
and
$$x=2n-2k-2+k(p+2-\frac{2p}{q}),\,\,\,\,\,\, y=2n-2k-2.$$
Then,
{\small
\begin{align}\label{0704-1}
x+2-k(p+2-\frac{2p}{q})=2n-2k\,\,  \mbox{ and }\,\,\frac{2qx}{2q+pq-2p}+\frac{(pq-2p)y}{2q+pq-2p}=2n-2.
\end{align}}
Since $T_u\subset S_u$ and $\om(T_u)\approx \om(S_u)$, Fubini's Theorem and H\"older's inequality imply
\begin{align}
J_{f,h}
\leq& \int_\D   |f^{(k)}(u)|^2|g^{(n-k)}(u)|^2\left(1-|u|\right)^{2n-2}\left(\int_{T_u}|h(z)|\om(z)dA(z)\right)dA(u)  \nonumber\\
\leq& \int_\D   |f^{(k)}(u)|^2|g^{(n-k)}(u)|^2\left(1-|u|\right)^{2n-2}(M_\om h)(u)\om(S_u)dA(u)  \nonumber\\
\leq& \left(\int_\D |f^{(k)}(u)|^{2+p-\frac{2p}{q}}|g^{(n-k)}(u)|^2(1-|u|)^x\om(S_u)dA(u)\right)^\frac{2q}{(2+p)q-2p} \nonumber\\
&\cdot \left(\int_\D |(M_\om h)(u)|^{1+\frac{2q}{p(q-2)}}|g^{(n-k)}(u)|^2(1-|u|)^y\om(S_u)dA(u)\right)^\frac{pq-2p}{(2+p)q-2p}.\label{0514-7}
\end{align}
Let $\mu_{x,k}$ and $\mu_{y,k}$ be defined as in (\ref{0424-1}).
Then, for any fixed $0<r<1$ and $z\in \D$, similarly to get (\ref{0514-4}), we have
\begin{align}\label{0514-6}
\frac{\mu_{x,k}(\D(z,r))}{(1-|z|)^{k(2+p-\frac{2p}{q})}\om(S_z)^\frac{pq+2q-2p}{pq}}
\lesssim \sup_{a\in\D} \frac{|g^{(n-k)}(a)|^2(1-|a|)^{x+2}\om(S_a)}{(1-|a|)^{k(2+p-\frac{2p}{q})}\om(S_a)^\frac{pq+2q-2p}{pq}}<\infty
\end{align}
and
\begin{align*}
\frac{\mu_{y,k} (S_a)}{\om(S_a)^{\frac{pq+2q-2p}{pq-2p}\slash\frac{q}{q-2}}}
&\lesssim \frac{\int_{S_a}\om(S_u)^{1+\frac{2}{p}-\frac{2}{q}}(1-|u|)^{-2}dA(u)}{\om(S_a)^{\frac{pq+2q-2p}{pq}}} \lesssim 1 .
\end{align*}
Here, the last inequality  comes from the   proof of \cite[Proposition 4.7]{PjaRj2014book}, i.e.,
$$\int_{S_a}\om(S_u)^{1+2\alpha}(1-|u|)^{-2}dA(u)\lesssim \om(S_a)^{1+2\alpha}, \,\,\mbox{ for any given }\,\,\alpha>0.$$
Thus, Lemma \ref{0516-1} and \cite[Corollary 2.2]{PjaRj2014book} deduce that  both $D^k: A_\om^p \to L_{\mu_{x,k}}^{2+p-\frac{2p}{q}}$ and  $M_\om: L_\om^\frac{q}{q-2}\to L_{\mu_{y,k}}^\frac{pq+2q-2p}{pq-2p} $ are  bounded.
Therefore,
$$J_{f,h}\lesssim \|f\|_{A_\om^p}^2\|h\|_{L_\om^{\frac{q}{q-2}}}.$$
By Lemma \ref{0413-2}, we have
\begin{align}
\|T_g^{n,k} f\|_{A_\om^q}^q
&\approx
\int_\D \left(\int_{\Gamma_z}|f^{(k)}(u)|^2|g^{(n-k)}(u)|^2\left(1-\frac{|u|}{|z|}\right)^{2n-2}dA(u)\right)^\frac{q}{2}\om(z)dA(z) \nonumber\\
&\leq  \sup_{\|h\|_{L_\om^\frac{q}{q-2}}\leq 1}  (J_{f,h})^\frac{q}{2} \label{0514-8}\\
&\lesssim \|f\|_{A_\om^p}^q.\nonumber
\end{align}
So, $T_g^{n,k}:A_\om^p\to A_\om^q$ is bounded.

{\it Case 3: $0<q<2$.} Suppose that (\ref{0413-1}) holds for some $n\in\NN$ and $k=0$.
For any $z\in\D$ and $f\in H(\D)$,
by Lemma \ref{0413-2}, H\"older's inequality and Fubini's theorem,  we have
{\small
\begin{align}
\|T_g^{n,0} f\|_{A_\om^q}^q
&\approx
\int_\D \left(\int_{\Gamma_z}|f(u)|^2|g^{(n)}(u)|^2\left(1-\frac{|u|}{|z|}\right)^{2n-2}dA(u)\right)^\frac{q}{2}\om(z)dA(z)  \nonumber\\
&\leq  \int_\D |(N f)(z)|^\frac{p(2-q)}{2}
\left(\int_{\Gamma_z}|f(u)|^{2+p-\frac{2p}{q}}|g^{(n)}(u)|^2\left(1-\frac{|u|}{|z|}\right)^{2n-2}dA(u)\right)^\frac{q}{2}\om(z)dA(z)\nonumber\\
&\leq \|f\|_{L_\om^p}^\frac{p(2-q)}{2}
\left(\int_\D \left(\int_{\Gamma_z}|f(u)|^{2+p-\frac{2p}{q}}|g^{(n)}(u)|^2\left(1-|u|\right)^{2n-2}dA(u)\right)\om(z)dA(z)\right)^\frac{q}{2}\nonumber\\
&\approx \|f\|_{A_\om^p}^\frac{p(2-q)}{2}
\left(\int_\D |f(u)|^{2+p-\frac{2p}{q}}|g^{(n)}(u)|^2\left(1-|u|\right)^{2n-2}\om(S_u)dA(u)\right)^\frac{q}{2}.\label{0514-10}
\end{align}
}
Let $\mu_{2n-2,0}$ be defined as in (\ref{0424-1}). For any $z\in\D$ and fixed $r\in (0,1)$, as we get  (\ref{0514-4}),  we obtain
\begin{align}\label{0514-9}
\frac{\mu_{2n-2,0}(\D(z,r))}{\om(S_z)^{\frac{2q+pq-2p}{pq}}}
\lesssim \sup_{a\in\D} \frac{|g^{(n)}(a)|^2\left(1-|a|\right)^{2n}\om(S_a)}{\om(S_a)^{\frac{2q+pq-2p}{pq}}}<\infty.
\end{align}
Then, Lemma \ref{0516-1} implies
$$\|T_g^{n,0} f\|_{A_\om^q}^q\lesssim \|f\|_{A_\om^p}^q.$$
That is, $T_g^{n,0}:A_\om^p\to A_\om^q(n=1,2,\cdots)$ is bounded.

By Proposition 5.1 in \cite{zhu}, for any given $N \in\N$,
$$(1-|a|)|g^{(N+1 )}(a)|\lesssim \sup_{|z|=\frac{|a|+1}{2}}|g^{(N )}(z)|.$$
So, when $k=1$ and $n=2,3,\cdots$, from
\begin{align*}
T_g^{n,1}f(z)= T_g^{n-1,0}f(z)-T_g^{n,0}f(z)-\frac{f^{(0)}(0)g^{(n-1)}(0)}{(n-1)!}z^{n-1},
\end{align*}
we see that  $T_g^{n,1}:A_\om^p\to A_\om^q(n=2,3,\cdots)$ are all bounded.

Then, by mathematical induction and
$$T_g^{n,k}f(z)= T_g^{n-1,k-1}f(z)-T_g^{n,k-1}f(z)-\frac{f^{(k-1)}(0)g^{(n-k)}(0)}{(n-1)!}z^{n-1}, 2\leq k<n,$$
we get the desired result.

%Then, assume that (\ref{0413-1}) implies $T_g^{n,k}:A_\om^p\to A_\om^q$ is bounded when
%either $0\leq k<n\leq N$  or
%$$n=N+1, k=0,1,\cdots, K, \mbox{ where } 0\leq K<N.$$
%Suppose (\ref{0413-1}) holds for $n=N+1$ and $k=K+1$. For any $f\in A_\om^p$, integrating by parts, we have
%\begin{align*}
%T_g^{N+1,K+1}f(z)= T_g^{N,K}f(z)-T_g^{N+1,K}f(z)-\frac{f^{(K)}(0)g^{(N-K)}(0)}{N!}z^N.
%\end{align*}
%By the assumption, $T_g^{N,K}:A_\om^p\to A_\om^q$ is bounded. By Proposition 5.1 in \cite{zhu},
%$$(1-|a|)|g^{(N+1-K)}(a)|\lesssim \sup_{|z|=\frac{|a|+1}{2}}|g^{(N-K)}(z)|.$$
%Then, (\ref{0413-1}) deduces
%\begin{align*}
%\frac{(1-|a|)^{N+1-K}|g^{(N+1-K)}(a)|}{\om(S_a)^{\frac{1}{p}-\frac{1}{q}}}
%&\lesssim \frac{(1-|a|)^{N-K}}{\om(S_a)^{\frac{1}{p}-\frac{1}{q}}}\sup_{|z|=\frac{|a|+1}{2}}|g^{(N-K)}(z)|\\
%&\lesssim\sup_{z\in\D}  \frac{(1-|z|)^{N-K}|g^{(N-K)}(z)|}{\om(S_z)^{\frac{1}{p}-\frac{1}{q}}}<\infty
%\end{align*}
%and $T_g^{N+1,K}:A_\om^p\to A_\om^q$ is bounded. So, $T_g^{N+1,K+1}:A_\om^p\to A_\om^q$ is bounded.

{\it (ii).}  Suppose that $T_g^{n,k}:A_\om^p\to A_\om^q$ is compact. By Lemma \ref{0425-1},  $\{F_a/ \|F_a\|_{A_\om^p} \}$ is bounded and converges to 0 uniformly on compact subsets of $\D$. From (\ref{0417-2}) and Lemma  \ref{0425-2}, we see that  (\ref{0514-3}) holds.

Conversely, suppose that (\ref{0514-3}) holds.
Let  $\{f_j\}$ be bounded in $A_\om^p$ and converge to 0 uniformly on compact subsets in $\D$.
Bearing in mind that $\mu_{\beta,k}$ is defined by (\ref{0424-1}).

If $q=2$, similarly to get (\ref{0514-4}), we have $\delta_{1,s}\to 0$ as $s\to 1^-$,  where
$$\delta_{1,s}=\sup_{|z|>s}\frac{\mu_{2n-2,k}(\D(z,r))}{(1-|z|)^{2k}\om(S_z)^\frac{2}{p}}.$$
By (\ref{0514-5}) and Lemma \ref{0516-1}, we have
\begin{align*}
\lim_{j\to\infty}\|T_g^{n,k}f_j\|_{A_\om^2}^2
&\lesssim  \lim_{j\to\infty}\left(\int_{s\D}+\int_{\D\backslash s\D}\right) |f_j^{(k)}(u)|^2|g^{(n-k)}(u)|^2(1-|u|)^{2n-2}\om(S_u)dA(u)\\
&\lesssim \delta_{1,s} \sup_{j\in\N} \|f_j\|_{A_\om^p}^2.
\end{align*}
Let $s\to 1$, $\lim\limits_{j\to\infty}\|T_g^{n,k}f_j\|_{A_\om^2}^2=0$. By Lemma \ref{0425-2}, $T_g^{n,k}:A_\om^p\to A_\om^2$ is compact.

If $2<q<\infty$,  similarly to get (\ref{0514-6}), we have $\delta_{2,s}\to 0$ as $s\to 1^-$,  where
$$\delta_{2,s}=\sup_{|z|>s}\frac{\mu_{x,k}(\D(z,r))}{(1-|z|)^{k(2+p-\frac{2p}{q})}\om(S_z)^\frac{pq+2q-2p}{pq}}.$$
Here, $x$ is given in (\ref{0704-1}).
By (\ref{0514-7}) and the boundedness of  $M_\om: L_\om^\frac{q}{q-2}\to L_{\mu_{y,k}}^\frac{pq+2q-2p}{pq-2p} $, we have
\begin{align*}
J_{f_j,h}
\lesssim \|h\|_{L_\om^\frac{q}{q-2}} \left(\int_\D |f_j^{(k)}(u)|^{2+p-\frac{2p}{q}}|g^{(n-k)}(u)|^2(1-|u|)^x\om(S_u)dA(u)\right)^\frac{2q}{(2+p)q-2p}.
\end{align*}
Here, $M_\om$ and $J_{f_j,h}$ are defined in (\ref{0704-2}) and (\ref{0704-3}), respectively.
By (\ref{0514-8}), we have
\begin{align*}
\|T_g^{n,k}f_j\|_{A_\om^q}\lesssim \left(\left(\int_{s\D}+\int_{\D\backslash s\D}\right) |f_j^{(k)}(u)|^{2+p-\frac{2p}{q}}|g^{(n-k)}(u)|^2(1-|u|)^x\om(S_u)dA(u)\right)^\frac{q}{(2+p)q-2p}.
\end{align*}
Therefore, Lemma \ref{0516-1} deduces
\begin{align*}
\lim_{j\to\infty} \|T_g^{n,k}f_j\|_{A_\om^q}
&\lesssim\lim_{j\to\infty} \left(\int_{\D\backslash s\D} |f_j^{(k)}(u)|^{2+p-\frac{2p}{q}}|g^{(n-k)}(u)|^2(1-|u|)^x\om(S_u)dA(u)\right)^\frac{q}{(2+p)q-2p}  \\
&\lesssim \delta_{2,s}^\frac{q}{(2+p)q-2p} \sup_{j\in\N} \|f_j\|_{A_\om^p}.
\end{align*}
Let $s\to 1$. We obtain that $\lim_{j\to\infty}\|T_g^{n,k}f_j\|_{A_\om^q}=0$. By Lemma \ref{0425-2}, $T_g^{n,k}:A_\om^p\to A_\om^q$ is compact.

If $0<q<2$ and $k=0$, similarly to get (\ref{0514-9}), we have $\delta_{3,s}\to 0$ as $s\to 1^-$, where
$$\delta_{3,s}=\sup_{|z|>s}\frac{\mu_{2n-2,0}(\D(z,r))}{\om(S_z)^{\frac{2q+pq-2p}{pq}}}.$$
By (\ref{0514-10})   and Lemma \ref{0516-1}, we have
{\small
\begin{align*}
\lim_{j\to\infty} \|T_g^{n,0} f_j\|_{A_\om^q}
&\lesssim  \sup_{j\in\N} \|f_j\|_{A_\om^p}^\frac{p(2-q)}{2q}
\lim_{j\to\infty}\left(\int_{\D\backslash s\D} |f_j(u)|^{2+p-\frac{2p}{q}}|g^{(n)}(u)|^2\left(1-|u|\right)^{2n-2}\om(S_u)dA(u)\right)^\frac{1}{2}\\
&\lesssim   \delta_{3,s}^{\frac{1}{2}}  \sup_{j\in\N} \|f_j\|_{A_\om^p}.
\end{align*}
}
Letting $s\to 1$, by Lemma \ref{0425-1}, $T_g^{n,0}:A_\om^p\to A_\om^q$ is compact.

Then, by mathematical induction, (\ref{0514-3}) deduces the compactness of $T_g^{n,k}:A_\om^p\to A_\om^q$ for all $n\in\N$ and $k=0,1,2,\cdots,n-1$.
The proof is complete.
\end{proof}

\begin{theorem}
Suppose $\om\in\hD$, $g\in H(\D)$,  $0<p<\infty$ and    $k,n\in\mathbb{Z}$ satisfying $0\leq k<n$.
\begin{enumerate}[(i)]
  \item  If $k\geq 1$, $T_g^{n,k}:A_\om^p\to A_\om^p$ is bounded (compact) if and only if $g\in \B$($g\in \B_0$).
  \item $T_g^{n,0}:A_\om^p\to A_\om^p$ is bounded (compact) if and only $g\in \mathcal{C}^1(\om^*)$($g\in \mathcal{C}_0^1(\om^*)$).
\end{enumerate}
\end{theorem}
\begin{proof}
{\it (i).} It is well known that, for any given $k\in\N$ and $g\in H(\D)$,
\begin{align}\label{0417-3}
|g(0)|+\sup_{z\in\D}(1-|z|^2)|g^\p(z)|\approx \sum_{j=0}^{k-1}|g^{(j)}(0)|+\sup_{z\in\D} (1-|z|^2)^k|g^{(k)}(z)|.
\end{align}
Suppose $g\in\B$. If $f(0)=f^\p(0)=\cdots=f^{(k-1)}(0)=0$, by Lemma \ref{0417-1}, we have
\begin{align*}
\|f\|_{A_\om^p}^p
&\approx \int_0^1 \|f_r\|_{H^p}^p \om(r)dr %\approx \int_0^1\|(f^{(k)})_r\|_{H^p}^p\om(r)dr\\
\approx \int_\D \left(\int_0^1 |f^{(k)}(tz)|^2(1-t)^{2k-1}dt\right)^\frac{p}{2}\om(z)dA(z)
\end{align*}
and
\begin{align*}
\|T_g^{n,k}f\|_{A_\om^p}^p
&\approx \int_\D \left(\int_0^1 |f^{(k)}(tz)|^2|g^{(n-k)}(tz)|^2(1-t)^{2n-1}dt\right)^\frac{p}{2}\om(z)dA(z)\\
&\leq \|g\|_{\B}^p \int_\D \left(\int_0^1 |f^{(k)}(tz)|^2(1-t)^{2k-1}dt\right)^\frac{p}{2}\om(z)dA(z)
\approx \|g\|_{\B}^p\|f\|_{A_\om^p}^p.
\end{align*}
That is, $T_g^{n,k}:A_\om^p\to A_\om^p$ is bounded.

The necessity can be obtained by   (\ref{0417-2}) and (\ref{0417-3}) directly.

{\it (ii).} Suppose $g\in\mathcal{C}^1(\om^*)$. By Theorem A, we have that $g\in\B$ and $T_g^{1,0}:A_\om^p\to A_\om^p$ is bounded.
When $n\geq 2$, for any $f\in A_\om^p$, using statement (i) and the fact that
\begin{align}\label{0417-4}
T_g^{1,0} f(z)=T_g^{n,0}f(z)+\sum_{k=1}^{n-1} C_{n-1}^k T_g^{n,k}f(z)+\sum_{k=1}^{n-1}\frac{1}{k!}\left(\sum_{j=0}^{k-1}C_{k-1}^j f^{(j)}(0)g^{(k-j)}(0)\right)z^k,
\end{align}
$T_g^{n,0}:A_\om^p\to A_\om^p$ is bounded.

Conversely, if $n\geq 2$ and $T_g^{n,0}:A_\om^p\to A_\om^p$ is bounded, by (\ref{0417-2}), $g\in\B$.
By statement (i) and (\ref{0417-4}), $T_g^{1,0}:A_\om^p\to A_\om^p$ is bounded. From Theorem A, $g\in\mathcal{C}^1(\om^*)$.

The compactness of $T_g^{n,k}$ can be characterized by modifying the proof above in a standard way
 and  we omit the details.
The proof is complete.
\end{proof}

\begin{theorem}
Suppose $g\in H(\D)$,  $0<q<p<\infty$, $\om\in\hD$ and  $k,n\in\mathbb{Z}$ satisfying $0\leq k<n$.
 \begin{enumerate}[(i)]
   \item If $g\in A_\om^{\frac{pq}{p-q}}$,  then $T_g^{n,k}:A_\om^p\to A_\om^q$ is compact.
   \item If $q\geq 2$ and $T_g^{n,0}:A_\om^p\to A_\om^q$ is bounded, then $g\in A_\om^{\frac{pq}{p-q}}$.
 \end{enumerate}
\end{theorem}
\begin{proof} {\it (i).} Suppose $g\in A_\om^\frac{pq}{p-q}$. %Without loss of generality, assume $\om(z)=0$ when $|z|\leq \frac{1}{2}$.
 For any $f\in A_\om^p$, by Lemma \ref{0413-2} and H\"older's inequality,
{\small
\begin{align}
\|T_g^{n,0}f\|_{A_\om^q}^q
\approx& \int_\D \left(\int_{\Gamma_z}|f(u)|^2|g^{(n)}(u)|^2\left(1-\frac{|u|}{|z|}\right)^{2n-2}dA(u)\right)^\frac{q}{2}\om(z)dA(z) \nonumber\\
\lesssim&  \int_\D |(Nf)(z)|^q\left(\int_{\Gamma_z}|g^{(n)}(u)|^2\left(1-\frac{|u|}{|z|}\right)^{2n-2}dA(u)\right)^\frac{q}{2}\om(z)dA(z)\label{0515-1}\\
\lesssim& \left(\int_\D |(Nf)(z)|^p\om(z)dA(z)\right)^\frac{q}{p}  \nonumber\\
&\cdot\left(\int_\D \left(\int_{\Gamma_z}|g^{(n)}(u)|^2\left(1-\frac{|u|}{|z|}\right)^{2n-2}dA(u)\right)^\frac{pq}{2(p-q)}\om(z)dA(z)\right)^{1-\frac{q}{p}}\label{0515-2}\\
\lesssim& \|f\|_{A_\om^p}^q\|g\|_{A_\om^{\frac{pq}{p-q}}}^q. \nonumber
\end{align}
}
So, $T_g^{n,0}:A_\om^p\to A_\om^q$ is bounded and
\begin{align}\label{0423-1}
\|T_g^{n,0}\|_{A_\om^p\to A_\om^q}\lesssim \|g\|_{A_\om^{\frac{pq}{p-q}}}.
\end{align}
Suppose $\{f_j\}$ is bounded in $A_\om^p$ and converges to 0 uniformly on compact subsets of $\D$.
For any $0\leq r<1$, let
$$J_1(f_j,r)=\left(\int_{r\D} |(Nf_j)(z)|^p\om(z)dA(z)\right)^\frac{1}{p}$$
and
$$J_2(g,r)=\left(\int_{\D\backslash r\D} \left(\int_{\Gamma_z}|g^{(n)}(u)|^2\left(1-\frac{|u|}{|z|}\right)^{2n-2}dA(u)\right)^\frac{pq}{2(p-q)}\om(z)dA(z)\right)^\frac{p-q}{pq}.$$
%Since $J_2(g,0)\approx \|g\|_{A_\om^\frac{pq}{p-q}}$,
Obviously, $\lim\limits_{r\to 1} J_2(g,r)=0$.
By (\ref{0515-1}) and (\ref{0515-2}), we have
\begin{align*}
\lim_{j\to\infty}\|T_g^{n,0}f\|_{A_\om^q}^q
&\lesssim \lim_{j\to \infty}\left((J_1(f_j,r))^q\|g\|_{A_\om^{\frac{pq}{p-q}}}^q  +\|f_j\|_{A_\om^p}^q (J_2(g,r))^q \right)\\
&\lesssim \left(\sup_{j\in\N}\|f_j\|_{A_\om^p}^q\right)(J_2(g,r))^q.
\end{align*}
Letting $r\to 1$, by Lemma \ref{0425-1}, $T_g^{n,0}:A_\om^p\to A_\om^q$ is compact.

When $k=1$ and $n=2,3,\cdots$, since
\begin{align*}
T_g^{n,1}f(z)= T_g^{n-1,0}f(z)-T_g^{n,0}f(z)-\frac{f^{(0)}(0)g^{(n-1)}(0)}{(n-1)!}z^{n-1},
\end{align*}
we get the compactness of $T_g^{n,1}:A_\om^p\to A_\om^q(n=2,3,\cdots)$.

Then, mathematical induction and the fact
$$T_g^{n,k}f(z)= T_g^{n-1,k-1}f(z)-T_g^{n,k-1}f(z)-\frac{f^{(k-1)}(0)g^{(n-k)}(0)}{(n-1)!}z^{n-1}, 2\leq k<n,$$
deduce the desired result.

%Then, assume that $g\in A_\om^\frac{pq}{p-q}$ implies $T_g^{n,k}:A_\om^p\to A_\om^q$ is compact when
%either $0\leq k<n\leq N$  or
%$$n=N+1, k=0,1,\cdots, K, \mbox{ where } 0\leq K<N.$$
%For any $f\in A_\om^p$, interating by parts, we have
%\begin{align*}
%T_g^{N+1,K+1}f(z)= T_g^{N,K}f(z)-T_g^{N+1,K}f(z)-\frac{f^{(K)}(0)g^{(N-K)}(0)}{N!}z^N.
%\end{align*}
%So, $T_g^{N+1,K+1}$ is compact. That is, ({\it i}) holds.

{\it (ii).} In order to prove this result, we extend the proof of \cite[Proposition 4.8]{PjaRj2014book} to $T_g^{n,0}$.
Suppose $0<\hat{p},\hat{q}<\infty$ satisfying $\frac{1}{\hat{q}}-\frac{1}{q}=\frac{1}{\hat{p}}-\frac{1}{p}>0$.
Let $\frac{1}{s}=\frac{1}{\hat{q}}-\frac{1}{q}$.
From \cite{KtRj2019mz}, for any $f\in A_\om^{\hat{p}}$, there exist $f_1\in A_\om^p$ and $f_2\in A_\om^s$ such that
$$f=f_1 f_2 \mbox{ and } \|f_1\|_{A_\om^p}\|f_2\|_{A_\om^s}\lesssim \|f\|_{A_\om^{\hat{p}}}.$$
Letting  $F=T_g^{n,0}f_1$, we have
$$T_g^{n,0}f=T_F^{n,0}f_2,  \,\,\,\mbox{ and }\,\,\,  \|F\|_{A_\om^q}\lesssim \|f_1\|_{A_\om^p}.$$
From (\ref{0423-1}),
\begin{align*}
\|T_{g}^{n,0}f\|_{A_\om^{\hat{q}}}=\|T_F^{n,0}f_2\|_{A_\om^{\hat{q}}}
%\lesssim \|T_{F}^{n,0}\|_{A_\om^s\to A_\om^{\hat{q}}}\|f_2\|_{A_\om^s}
\lesssim \|F\|_{A_\om^q} \|f_2\|_{A_\om^s}
\lesssim \|f_1\|_{A_\om^p} \|f_2\|_{A_\om^s}
\lesssim \|f\|_{A_\om^{\hat{p}}}.
\end{align*}
That is, $T_g^{n,0}:A_\om^{\hat{p}}\to A_\om^{\hat{q}}$ is bounded.

%That is, assuming  $0<q<p<\infty$ and $\om\in\hD$,
%if $T_g^{n,0}:A_\om^p\to A_\om^p$ is bounded,
%then $T_g^{n,0}:A_\om^{\hat{p}}\to A_\om^{\hat{q}}$ is bounded whenever $\frac{1}{s}=\frac{1}{\hat{q}}-\frac{1}{q}=\frac{1}{\hat{p}}-\frac{1}{p}>0$.
%From \cite{KtRj2019mz}, for any $f\in A_\om^{\hat{p}}$, there exist $f_1\in A_\om^p$ and $f_2\in A_\om^s$ such that
%$$f=f_1 f_2 \mbox{ and } \|f_1\|_{A_\om^p}\|f_2\|_{A_\om^s}\lesssim \|f\|_{A_\om^{\hat{p}}}.$$
%Therefore, $T_g^{n,0}f=T_F^{n,0}f_2$ where $F=T_g^{n,0}f_1$. Then $\|F\|_{A_\om^q}\lesssim \|f_1\|_{A_\om^p}$.
%From (\ref{0423-1}),
%\begin{align*}
%\|T_{g}^{n,0}f\|_{A_\om^{\hat{q}}}=\|T_F^{n,0}f_2\|_{A_\om^{\hat{q}}}
%\lesssim \|T_{F}^{n,0}\|_{A_\om^s\to A_\om^{\hat{q}}}\|f_2\|_{A_\om^s}
%&\lesssim \|F\|_{A_\om^q} \|f_2\|_{A_\om^s}\\
%&\lesssim \|f_1\|_{A_\om^p} \|f_2\|_{A_\om^s}
%\lesssim \|f\|_{A_\om^{\hat{p}}}.
%\end{align*}

So, if $q\geq 2$ and $T_g^{n,0}:A_\om^p\to A_\om^q$ is bounded,  $T_g^{n,0}:A_\om^\frac{2pq}{pq-2(p-q)}\to A_\om^2$ is  also bounded.
Let $\chi_{\D\backslash \frac{1}{2}\D}(z)$ be the characteristic function on $\D\backslash\frac{1}{2}\D$.
By a calculation, $S_{\frac{1+|z|}{2|z|}z}\subset T_z$ when $|z|\geq \frac{1}{2}$.
From Lemma \ref{0413-2} and Fubini's theorem,  for any $f\in A_\om^\frac{2pq}{pq-2(p-q)}$, we get
\begin{align*}
\|T_g^{n,0}f\|_{A_\om^2}^2
&\approx \int_\D \int_{\Gamma_\xi} |f(z)|^{2}|g^{(n)}(z)|^{2}(1-\frac{|z|}{|\xi|})^{2n-2}dA(z)\om(\xi)\chi_{\D\backslash\frac{1}{2}\D}(\xi)dA(\xi)  \\
&\gtrsim \int_{\D\backslash\frac{3}{4}\D} |f(z)|^{2}|g^{(n)}(z)|^{2}\int_{T_z} (|\xi|-|z|)^{2n-2}\om(\xi)dA(\xi)dA(z)\\
&\geq \int_{\D\backslash\frac{3}{4}\D} |f(z)|^{2}|g^{(n)}(z)|^{2}\int_{S_{\frac{1+|z|}{2|z|}z}} (|\xi|-|z|)^{2n-2}\om(\xi)dA(\xi)dA(z)\\
&\approx \int_\D |f(z)|^{2}|g^{(n)}(z)|^{2} (1-|z|)^{2n-1}\hat{\om}(z)dA(z).
\end{align*}
Since $(1-|z|)\hat{\om}(z)\approx \om(S_z)$, letting
$\mu_{2n-2,0}$ be defined as in (\ref{0424-1}).
We obtain that $Id: A_\om^\frac{2pq}{pq-2(p-q)}\to L_{\mu_{2n-2,0}}^2$ is bounded. By  Lemma \ref{0516-1}, we have
$$
\|Id\|_{A_\om^\frac{2pq}{pq-2(p-q)}\to L_{\mu_{2n-2,0}}^2}^2
\approx\left(\int_\D \left(\int_{\Gamma_\xi}|g^{(n)}(z)|^2(1-|z|)^{2n-2}dA(z)\right)^\frac{pq}{2(p-q)} \om(\xi)dA(\xi)\right)^\frac{2(p-q)}{pq}<\infty.
$$
Therefore,
$$
\int_\D \left(\int_{\Gamma_\xi}|g^{(n)}(z)|^2(1-\frac{|z|}{|\xi|})^{2n-2}dA(z)\right)^\frac{pq}{2(p-q)} \om(\xi)dA(\xi)<\infty.
$$
From Lemma \ref{0413-2},  $g\in A_\om^\frac{pq}{p-q}$ and $\|g\|_{A_\om^\frac{pq}{p-q}}\lesssim \|T_g^{n,0}\|_{A_\om^p\to A_\om^q}$.
%
%
%If $q<2$ and $\frac{1}{q}-\frac{1}{p}<\frac{1}{2}$, there exist $j\in\N$  and $\hat{q}>2$ such that $\frac{1}{q}-\frac{1}{p}=\frac{1}{\hat{q}}-\frac{1}{j p}$.
% For any $f_1\in A_\om^{\frac{jp}{j-1}}$ and $f_2\in A_\om^{jp}$, let $F=T_g^{n,0}f_1$.
% From Lemma \ref{0413-2} and H\"older's inequality, we have
%\begin{align*}
%\|T_F f_2\|_{A_\om^{\hat{q}}}
%&=\|T_g f_1 f_2\|_{A_\om^{\hat{q}}}
%\lesssim \|T_g f_1 f_2\|_{A_\om^q}\leq \|T_g\|_{A_\om^p\to A_\om^q}\|f_1\|_{A_\om^{\frac{jp}{j-1}}}\|f_2\|_{A_\om^{jp}}.
%\end{align*}
%Then,
%$$\|T_g f_1\|_{A_\om^\frac{pq}{p-q}}=\|F\|_{A_\om^\frac{pq}{p-q}}\lesssim \|T_g\|_{A_\om^p\to A_\om^q}\|f_1\|_{A_\om^{\frac{jp}{j-1}}}.$$
%Let $f_1\equiv 1$, we get the desired, i.e. $\|g\|_{A_\om^\frac{pq}{p-q}}\lesssim \|T_g\|_{A_\om^p\to A_\om^q}.$
\end{proof}

\section{Generalized Toeplitz operators}
In order to characterize the Schatten class operator $T_g^{n,k}$, we define a new kind of Toeplitz operators $\T^\om_{\mu,k}$, which will be called generalized Toeplitz operators.
%, on the Bergman space $A_\om^2$.
%Suppose $k=0,1,2,\cdots$ and $\mu$ is  a positive Borel measure on $\D$.
Let $B_z^\om$ be the reproducing kernel of $A_\om^2$, that is, for any $f\in A_\om^2$ and $z\in\D$, we have
$$f(z)=\langle f,B_z^\om \rangle_{A_\om^2}=\int_\D f(w)\ol{B_z^\om(w)}\om(w)dA(w).$$
Let $\om_j=\int_0^1 t^j\om(t)dt$. By using the standard orthonormal basis $\{z^j/\sqrt{2\om_{2j+1}}\}, j\in\N\cup\{0\}$, of $A_\om^2$, we obtain
$$B_z^\om(w)=\sum_{j=0}^\infty \frac{(\ol{z}w)^j}{2\om_{2j+1}}.$$

If $k=0,1,2,\cdots$, let $D^k$ be the $k$-th differential operator and $B_z^{\om,k}=D^k B_z^\om$.
The generalized Toeplitz operator $\T_{\mu,k}^\om$ is defined by
$$\T_{\mu,k}^\om f(z)=\int_\D f^{(k)}(w)\ol{ B_z^{\om,k}(w)}d\mu(w).$$
Obviously, if $k=0$, $\T_{\mu,k}^{\om}$ is the Toeplitz operator $\T_\mu^{\om}$, which was   studied in \cite{PjaRj2016,PjaRjSk2018jga}.
 % the Schatten class Toeplitz operators $\T_\mu^\om$ on $A_\om^2$ were completely characterized.
% Using these results, we will completely characterized the Schatten class generalized Toeplitz operators on $A_\om^2$.
%Let $\mcS_p(A_\om^2)$ be the Schatten $p$-class operators on $A_\om^2$. When $k>0$,  we find a $\eta\in\R$ such that  $\eta(t)\approx(1-t)^{2k-1}\hat{\om}(t)$ and $\T_\mu^{\om,k}\in\mcS_p(A_\om^2)$ if and only if $\T_\mu^\eta\in \mcS_p(A_\eta^2)$.

\begin{lemma}\label{0510-3}
Suppose $\om\in\hD$ and $k\in\N$. There exists a regular weight $\upsilon$ such that  $\upsilon(t)\approx (1-t)^{k-1}\hat{\om}(t)$ and
\begin{align}\label{0510-1}
\sum_{j=k}^\infty \frac{j!}{(j-k)!}\frac{(w\ol{z})^{j-k}}{2\om_{2j+1}}=B_z^\upsilon(w), \mbox{ for all } z,w\in\D.
\end{align}
\end{lemma}
\begin{proof}
 Let $k=1$. For $j=1,2,\cdots$, we have
\begin{align*}
\frac{(j-1)!\om_{2j+1}}{j!}
&=\frac{1}{j}\int_0^1 t^{2j}d\left(-\int_t^1 s\om(s)ds\right)
=\int_0^1 t^{2j-1}\left(2\int_t^1 s\om(s)ds\right)dt.
\end{align*}
So, (\ref{0510-1}) holds by letting $\upsilon(t)=2\int_t^1 s\om(s)ds$.

From Lemmas 1.6 and 1.7 in \cite{PjaRj2014book}, we have
$$\upsilon(t)\approx \hat{\om}(t)\,\, \mbox{ and }\,\, \upsilon\in\R.$$
So, the lemma holds when $k=1$.

Suppose the lemma holds for $k=K$, that is,  there is a regular weight $\eta$ such that $\eta(t)\approx(1-t)^{K-1}\hat{\om}(t)$ and $$\sum_{j=K}^\infty \frac{j!}{(j-K)!}\frac{(w\ol{z})^{j-K}}{2\om_{2j+1}}=B_z^\eta(w), \mbox{ for all } z,w\in\D.$$
Then,
{\small
\begin{align*}
\sum_{j=K+1}^\infty \frac{j!}{(j-K-1)!}\frac{(w\ol{z})^{j-K-1}}{2\om_{2j+1}}
=\left.\frac{d}{dt} \sum_{j=K}^\infty \frac{j!}{(j-K)!}\frac{t^{j-K}}{2\om_{2j+1}}\right|_{t=w\ol{z}}
=\left.\frac{d}{dt}\sum_{j=0}^\infty \frac{t^j}{2\eta_{2j+1}}\right|_{t=w\ol{z}}.
\end{align*}
}
By the proof of the case $k=1$, there is a regular weight $\tau$ such that
$$\tau(t)\approx \hat{\eta}(t)\approx (1-t)^K\hat{\om}(t)$$
and
$$\left.\frac{d}{dt}\sum_{j=0}^\infty \frac{t^j}{2\eta_{2j+1}}\right|_{t=w\ol{z}}
=\sum_{j=1}^\infty \frac{j!}{(j-1)!}\frac{(w\ol{z})^{j-1}}{2\eta_{2j+1}}
=B_z^\tau(w), \mbox{ for all } z,w\in\D.$$
So, the lemma holds when $k=K+1$.
Thus, the mathematical induction implies the desired result. The proof is complete
%
%
%
% When $k=K+1$
%
%Meanwhile, in this case, (\ref{0510-1}) equals to say
%\begin{align}\label{0510-2}
%\left.\frac{d}{dt}\sum_{j=0}^\infty \frac{t^j}{2\om_{2j+1}}\right|_{t=w\ol{z}}
%=\left.\sum_{j=0}^\infty \frac{t^j}{2\upsilon_{2j+1}}\right|_{t=w\ol{z}}.
%\end{align}
%So, mathematical induction implies the statement holds for any given $k\in\N$. The proof is complete.
\end{proof}

A sequence $\{a_j\}_{j=1}^\infty$ is a $r$-lattice for some $0<r<\infty$ means
$\beta(a_i,a_j)\geq \frac{r}{5}$ for all $i\neq j$ and $\D=\cup_{j=1}^\infty \D(a_j,5r)$. For convenience, let $M_k:H(\D)\to H(\D)$ be the multiplier $M_kf(z)=z^kf(z)$.\msk

Next, we will characterize the boundedness, compactness and Schatten class of generalized Toeplitz operators $\T_{\mu,k}^\om$ on $A_\om^2$.

\begin{theorem}\label{0514-1}
Suppose $k\in\N$, $0<p<\infty$ and $\om\in\hD$. Let $0<r<\infty$ and $\{a_j\}_{j=1}^\infty$ be a $r$-lattice of $\D$.
\begin{enumerate}[(i)]
  \item $\T^\om_{\mu,k}$ is bounded on $A_\om^2$ if and only if  $\sup\limits_{z\in\D}\frac{\mu(\D(z,r))}{(1-|z|)^{2k+1}\hat{\om}(z)}<\infty$;
  \item  $\T^\om_{\mu,k}$ is compact on $A_\om^2$ if and only if   $\lim\limits_{|z|\to 1}\frac{\mu(\D(z,r))}{(1-|z|)^{2k+1}\hat{\om}(z)}=0$;
  \item  $\T^\om_{\mu,k}\in \mcS_p(A_\om^2)$ if and only if
  $\sum\limits_{j=1}^\infty \left(\frac{\mu(\D(a_j,5r))}{(1-|a_j|)^{2k+1}\hat{\om}(a_j)}\right)^p<\infty$.
\end{enumerate}
\end{theorem}
\begin{proof}
Using Lemma \ref{0510-3}, there is a  regular weight $\upsilon$  such that
\begin{align*}
\T^\om_{\mu,k}f(z)=z^k\int_\D f^{(k)}(w)\ol{B_z^\upsilon(w)}d\mu(w).
\end{align*}
Then, Lemma \ref{0413-2} implies
\begin{align*}
\|\T^\om_{\mu,k}f\|_{A_\om^2}^2
\approx&\int_\D \left|\int_\D f^{(k)}(w)\ol{B_z^\upsilon(w)}d\mu(w)\right|^2\om(z)dA(z) \\
\approx &\sum_{j=0}^{k-1}\left|\int_\D f^{(k)}(w)D^j{B_w^\upsilon(0)}d\mu(w)\right|^2  \\
&+\int_\D \left|\int_\D f^{(k)}(w) D^k {B_w^\upsilon(z)}d\mu(w)\right|^2(1-|z|)^{2k-1}\hat{\om}(z)dA(z).
\end{align*}
By Lemma \ref{0510-3}, there exists a $\eta\in\R$ such that
\begin{align}\label{0706-1}
D^k {B_w^\upsilon(z)}=\ol{w}^kB_w^\eta(z),\,\,\,\,\,
\eta(t)\approx (1-t)^{k-1}\hat{\upsilon}(t)\approx (1-t)^{2k-1}\hat{\om}(t).
 \end{align}
Then, we have%So, letting $c_j$ be the coefficient of $D^j B_w^\upsilon(0)$, we have
$$\|f\|_{A_\om^2}^2\approx \sum_{j=0}^{k-1}|f^{(j)}(0)|^2+\|f^{(k)}\|_{A_\eta^2}^2$$
and
\begin{align}
\|\T^\om_{\mu,k}f\|_{A_\om^2}^2
\approx &\sum_{j=0}^{k-1} \left|\int_\D f^{(k)}(w)\ol{w}^j d\mu(w)\right|^2  \nonumber\\
&+\int_\D \left|\int_\D f^{(k)}(w)  {B_w^\eta(z)}\ol{w}^k d\mu(w)\right|^2\eta(z)dA(z). \nonumber
\end{align}
So, $\T^\om_{\mu,k}$ is bounded(or compact) on $A_\om^2$ if and only if all the operators
$$T_j:A_\eta^2\to \CC, \,\,T_j f=\int_\D f(w)\ol{w}^jd\mu(w), \mbox{ for } j=0,1,\cdots,k-1,$$
and
$$T_k:A_\eta^2\to A_\eta^2,\,\,T_k f(z)=\int_\D f(w)\ol{w}^k B_w^\eta(z)d\mu(w)$$
are bounded(or compact).

{\it (i).}  Suppose $\T_{\mu,k}^\om:A_\om^2\to A_\om^2$ is bounded. From the boundedness of $T_0$, $\mu(\D)<\infty$.
Since $T_k$ is bounded on $A_\eta^2$, by Theorem C in \cite{PjaRjSk2018jga}, we have
\begin{align*}
\int_\D |w|^{2k} |B_z^\eta(w)|^2 d\mu(w)
&=(T_k M_k B_z^\eta)(z)
=\langle T_k M_k B_z^\eta, B_z^\eta \rangle_{A_\eta^2}\\
&\leq \|T_k\|_{A_\eta^2\to A_\eta^2} \|M_k\|_{A_\eta^2\to A_\eta^2}\|B_z^\eta\|_{A_\eta^2}^2
\approx \frac{\|T_k\|_{A_\eta^2\to A_\eta^2}}{(1-|z|)^2\eta(z)}.
\end{align*}
Without loss of generality, we can assume $r$ is that in \cite[Lemma 8]{PjaRjSk2018jga}. So, for any $|z|>\frac{1}{2}$,  $|B_z^\eta(w)|\approx B_z^\eta(z)$ when $w\in \D(z,r)$. Thus,
\begin{align}\label{0705-1}
\int_\D |w|^{2k} |B_z^\eta(w)|^2 d\mu(w)
\geq \int_{\D(z,r)} |w|^{2k}|B_z^\eta(w)|^2d\mu(w)
\approx \frac{\mu(\D(z,r))}{(1-|z|)^4(\eta(z))^2}.
\end{align}
Therefore,
$$\sup_{z\in\D} \frac{\mu(\D(z,r))}{(1-|z|)^{2k+1}\hat{\om}(z)}
\approx \sup_{z\in\D} \frac{\mu(\D(z,r))}{(1-|z|)^{2}\eta(z)}<\infty.$$

Conversely, suppose
\begin{align*}%\label{0513-3}
\sup_{z\in\D} \frac{\mu(\D(z,r))}{(1-|z|)^{2k+1}\hat{\om}(z)}<\infty.
\end{align*}
By Theorem 1 in \cite{PjaRjSk2018jga}, $\T_\mu^\eta$ is bounded on $A_\eta^2$.
Since $M_k$ is bounded on $A_\eta^2$, $M_k^*(\T_\mu^\eta)^*$ is bounded  on $A_\eta^2$.
For any $f\in A_\eta^2$, we have  $(\T_\mu^\eta)^*=\T_\mu^\eta$,
\begin{align*}
M_k^* f(z)
&=\langle M_k^* f, B_z^\eta\rangle_{A_\eta^2}
=\ol{\langle M_k B_z^\eta,f\rangle_{A_\eta^2}}
=\int_\D \ol{w}^k f(w) B_w^\eta(z)\eta(w)dA(w)
\end{align*}
and then
\begin{align}
M_k^*(\T_\mu^\eta)^*f(z)
&=\int_\D \ol{w}^k \left(\int_\D f(\xi) B_\xi^\eta(w)d\mu(\xi)\right) B_w^\eta(z)\eta(w)dA(w)  \nonumber\\
&=\int_\D f(\xi)\ol{\int_\D w^k B_w^\eta(\xi) B_z^\eta(w)    \eta(w)dA(w)}  d\mu(\xi)  \nonumber\\
&=\int_\D f(\xi) \ol{\xi^k  B_z^\eta(\xi)}d\mu(\xi) \nonumber\\
&=T_k f(z).  \label{0513-2}
\end{align}
Therefore, $T_k$  is bounded  on $A_\eta^2$.

When $j=0,1,\cdots,k-1$, since $\eta\in\R$,   for any $f\in A_\eta^2$, by H\"older's inequality and Lemma 2.5 in \cite{HzLj2018jga}, we have
\begin{align}\label{0827-1}
\left|T_j f\right|
&\lesssim \left(\int_{\D} |f(z)|^2d\mu(z)\right)^\frac{1}{2}
\lesssim \|f\|_{A_\eta^2}.
\end{align}
So, $T_j:A_\eta^2\to \CC$ is bounded. Therefore,  $\T^\om_{\mu,k}$ is bounded   on $A_\om^2$.

{\it (ii).}
Suppose $\T_{\mu,k}^\om:A_\om^2\to A_\om^2$ is compact.
Let $b_z^\eta=\frac{B_z^\eta}{\|B_z^\eta\|_{A_\eta^2}}$.
By Corollary 2 in \cite{PjRj2016jmpa}, $\|B_z^\eta\|_{A_\eta^2}^2\approx \frac{1}{(1-|z|)^2\eta(z)}$.
Then $\{b_z^\eta\}_{z\in\D}$ is bounded in $A_\eta^2$ and converges to 0 uniformly on compact subsets of $\D$ as $|z|$ approaches 1.
Replacing $B_z^\eta$ by $b_z^\eta$ and by Lemma \ref{0425-2} and (\ref{0705-1}), we get
$$\lim\limits_{|z|\to 1}\frac{\mu(\D(z,r))}{(1-|z|)^{2k+1}\hat{\om}(z)}=0.$$

Conversely, suppose
\begin{align*}%\label{0705-2}
\lim\limits_{|z|\to 1}\frac{\mu(\D(z,r))}{(1-|z|)^{2k+1}\hat{\om}(z)}=0.
\end{align*}
By Theorem 13 in \cite{PjaRjSk2018jga}, $\T_\mu^\eta$ is compact on $A_\eta^2$.
Since $M_k$ is bounded on $A_\eta^2$, $M_k^*(\T_\mu^\eta)^*$ is  compact  on $A_\eta^2$.
Thus, by (\ref{0513-2}), $T_k$  is  compact  on $A_\eta^2$.

Meanwhile, by Lemma 2.15 in \cite{HzLj2018jga}, $Id: A_\eta^2\to L_\mu^2$ is compact.
By Lemma \ref{0425-2} and (\ref{0827-1}), $T_j:A_\eta^2\to \CC$ is compact.
Therefore,  $\T^\om_{\mu,k}$ is  compact on $A_\om^2$.

%(\ref{0705-2}), for any given $\varepsilon>0$, there exists $\delta\in(0,1)$, such that
%$$\sup_{|z|>\delta}\frac{\mu(D(z,r))}{(1-|z|)^{2k+1}\hat{\om}(z)}<\varepsilon.$$
%When $j=0,1,\cdots,k-1$, since $\eta\in\R$,  for any $f\in A_\eta^2$, by H\"older's inequality and Theorem 4 in \cite{HzLj2018jga}, we have
%\begin{align*}
%\left|T_j f\right|^2
%&\lesssim \int_{\D} |f(z)|^2d\mu(z)
%=\left(\int_{\delta\D}+\int_{\D\backslash\delta\D}\right) |f(z)|^2d\mu(z)  \\
%&\lesssim  \int_{\delta\D}|f(z)|^2 d\mu(z)+\varepsilon \|f\|_{A_\eta^2}^2.
%\end{align*}
%Then, Lemma \ref{0425-2} deduces that $T_j:A_\eta^2\to \CC$ is compact.
%Therefore,  $\T^\om_{\mu,k}$ is  compact on $A_\om^2$.

{\it (iii).}
For any $f\in A_\om^2$, let
$$Tf(z)=\int_\D f^{(k)}(w)\ol{B_z^\upsilon(w)}d\mu(w).$$ %Then $T\in\mcS_p(A_\om^2)$ implies %$\T^\om_{\mu,k}\in\mcS_p(A_\om^2)$, since $M_k$ is bounded on $A_\om^2$ and $\mcS_p(A_\om^2)$ is an ideal. Conversely, suppose $\T^\om_{\mu,k}\in\mcS_p(A_\om^2)$.
Suppose that $\lambda_{j+1}$  is  the $j+1$-th singular value of $\T^\om_{\mu,k}$ on $A_\om^2$.
By Lemma \ref{0517-1}, we have
\begin{align*}
\lambda_{j+1}
&=\min_{f_1,f_2,\cdots,f_j\in A_\om^2}\max\left\{\|\T^\om_{\mu,k}f\|_{A_\om^2}:\|f\|_{A_\om^2}=1,\langle f,f_i\rangle_{A_\om^2}=0, 1\leq i\leq j\right\}.
\end{align*}
Since  $\T^\om_{\mu,k}=M_k T$ and $M_k$ is bounded and bounded below on $A_\om^2$,
\begin{align*}
\lambda_{j+1}
&\approx \min_{f_1,f_2,\cdots,f_j\in A_\om^2}\max\left\{\|Tf\|_{A_\om^2}:\|f\|_{A_\om^2}=1,\langle f,f_i\rangle_{A_\om^2}=0, 1\leq i\leq j\right\}.
\end{align*}
Therefore,
\begin{align}
\T^\om_{\mu,k}\in \mcS_p(A_\om^2) \Leftrightarrow T\in \mcS_p(A_\om^2). \label{0513-1}
\end{align}

For any $f\in H(\D)$, recall that $I f(z)=\int_0^z f(\xi)d\xi$ and $I^k$ is the $k$-th iteration of $I$.
By Lemma \ref{0413-2}, $I^k:A_\eta^2\to A_\om^2$ is bounded and bounded below.
For $j=0,1,2,\cdots$, let $\mathcal{F}_{j}(A_\om^2)$ be the set of all linear operators on $A_\om^2$ with rank less than or equals to $j$.
For any $f\in A_\om^2$ and $F\in \mathcal{F}_{j}(A_\om^2)$, by Lemma \ref{0413-2}, we have
\begin{align*}
\|(Tf)^{(k)}-(Ff)^{(k)}\|_{A_\eta^2}^2\lesssim \|(T-F)f\|_{A_\om^2}^2,
\end{align*}
which means
$$\|D^k T -D^k F\|_{A_\om^2\to A_\eta^2}\lesssim \|T-F\|_{A_\om^2\to A_\om^2},$$
and
$$\|D^k T I^k -D^k F I^k\|_{A_\eta^2\to A_\eta^2}\lesssim \|T-F\|_{A_\om^2\to A_\om^2}.$$
For any $F\in\mathcal{F}_j(A_\om^2)$, $D^k F I^k\in\mathcal{F}_{j}(A_\eta^2)$.
Since $D^k T I^k=T_k$, by Lemma \ref{0517-1},
\begin{align}\label{0517-2}
T\in \mcS_p(A_\om^2) \Rightarrow T_k\in\mcS_p(A_\eta^2).
\end{align}

For any $f\in H(\D)$, let $\hat{f}_j$ be the $j$-th Taylor coefficient  of $f$  and
$$A_\om^2(k)=\{f\in A_\om^2:\hat{f}_0=\cdots=\hat{f}_{k-1}=0\},$$
$$T_{+}f(z)
=\sum_{j=k}^\infty \frac{z^j}{2\upsilon_{2j+1}}\int_\D f^{(k)}(w)\ol{w^j}d\mu(w),$$
and
$$T_{-}f(z)
=\sum_{j=0}^{k-1} \frac{z^j}{2\upsilon_{2j+1}}\int_\D f^{(k)}(w)\ol{w^j}d\mu(w).
 $$
Then, we have:
\begin{enumerate}
  \item[(a)] $T=T_{+}+T_{-}$ and $T_{-}\in\mcS_p(A_\om^2)$;
  \item[(b)] $A_\om^2(k)$ is a Hilbert space under the  inner product $\langle \cdot,\cdot\rangle_{A_\om^2}$;
  \item[(c)] $\inf\left\{\|T_{+}-F\|_{A_\om^2\to A_\om^2}:F\in \mathcal{F}_j(A_\om^2)\right\}
  \leq  \inf\left\{\|T_{+}-F\|_{A_\om^2(k)\to A_\om^2(k)}:F\in \mathcal{F}_j(A_\om^2(k))\right\}$.
\end{enumerate}
 Here,  we only need to prove (c). For any $F\in \mathcal{F}_j(A_\om^2(k))$,  let $F^\p=FP_k$, in which
 $$(P_kf)(z)=\sum_{j=k}^\infty \hat{f}_j z^j,\,\,\,\mbox{ for all } \,\,f\in H(\D).$$
Then, $F^\p\in \mathcal{F}_j(A_\om^2)$  and
\begin{align*}
\|T_{+}-F^\p\|_{A_\om^2\to A_\om^2}
=\sup_{f\in A_\om^2} \frac{\|T_{+}f-F^\p f\|_{A_\om^2}}{\|f\|_{A_\om^2}}
\leq \sup_{f\in A_\om^2(k)} \frac{\|T_{+}f-F f\|_{A_\om^2}}{\|f\|_{A_\om^2}}.
\end{align*}
Therefore, (c) holds. By Lemma \ref{0517-1}, if $T_{+}\in\mcS_p(A_\om^2(k))$, $T_{+}\in \mcS_p(A_\om^2)$.

It is easy to check that both of $I^k:A_\eta^2\to A_\om^2(k)$  and $D^k: A_\om^2(k)\to A_\eta^2$ are   bounded and  bijective.
Let  $(I^k)^{-}$ and $(D^k)^-$  be the inverse operators of  $I^k:A_\eta^2\to A_\om^2(k)$  and $D^k: A_\om^2(k)\to A_\eta^2$, respectively.
For any $f\in A_\om^2(k)$, by (\ref{0706-1}),  we have
\begin{align*}
D^k T_+ I^k f(z)
&=\sum_{j=k}^\infty\frac{j!}{(j-k)!}\frac{z^{j-k}}{2\upsilon_{2j+1}}\int_\D f(w)\overline{w^j}d\mu(w)\\
&=\int_\D f(w)D^k B_w^\upsilon(z)d\mu(w)
=\int_\D \overline{w}^k f(w)B_w^\eta(z)d\mu(w)=T_k f(z).
\end{align*}
Thus, for any $F\in \mathcal{F}_j(A_\eta^2)$, we have $(D^k)^- F (I^k)^-\in \mathcal{F}_j(A_\om^2(k))$ and
\begin{align*}
\|T_+ - (D^k)^- F (I^k)^-\|_{A_\om^2(k)\to A_\om^2(k)}
&=\|(D^k)^- T_k (I^k)^- - (D^k)^- F (I^k)^-\|_{A_\om^2(k)\to A_\om^2(k)}\\
&\lesssim \| T_k - F \|_{A_\eta^2\to A_\eta^2}.
\end{align*}
By Lemma \ref{0517-1} and statement (a),  % $T_k\in \mcS_p(A_\eta^2)$ implies $T_+\in \mcS_p(A_\om^2(k))$. So, both $T_+$ and $T$ are in $\mcS_p(A_\om^2)$.
%That is,
\begin{align}\label{0517-3}
T_k\in \mcS_p(A_\eta^2)
\Rightarrow T_+\in \mcS_p(A_\om^2(k))
\Rightarrow T_+\in \mcS_p(A_\om^2)
\Rightarrow T\in \mcS_p(A_\om^2).
\end{align}
By (\ref{0513-2}),  (\ref{0513-1}), (\ref{0517-2}) and (\ref{0517-3}),  we have
\begin{align}\label{0706-2}
\T_\mu^\eta M_k \in\mcS_p(A_\eta^2)
\Leftrightarrow T_k\in\mcS_p(A_\eta^2)
\Leftrightarrow T\in \mcS_p(A_\om^2)
\Leftrightarrow \T^\om_{\mu,k}\in \mcS_p(A_\om^2).
\end{align}

For any $f=\sum_{j=0}^\infty \hat{f}_jz^j$, let $e_j(z)=z^j$ for $j=0,1,2\cdots$,
$$J_{k,1} f(z)=\sum_{j=k}^\infty\hat{f}_j z^{j-k}, \,\,J_{k,2} f(z)=\sum_{j=0}^{k-1} \hat{f}_j(\T_\mu^\eta e_j)(z).$$
Then we have the following statements:
\begin{enumerate}
  \item[(d)] $J_{k,1}$ is bounded on $A_\eta^2$;
  \item[(e)] the rank of $J_{k,2}$  is  finite  and therefore $J_{k,2}\in \mcS_p(A_\eta^2)$;
  \item[(f)] $\T_\mu^\eta= J_{k,2}+\T_\mu^\eta M_k J_{k,1}$.
\end{enumerate}
So, $\T_\mu^\eta M_k\in \mcS_p(A_\eta^2)$ implies $\T_\mu^\eta\in \mcS_p(A_\eta^2)$. Thus,
$$T_\mu^\eta M_k\in \mcS_p(A_\eta^2)\Leftrightarrow\T_\mu^\eta\in \mcS_p(A_\eta^2).$$
Then, (\ref{0706-2}) deduces that
$$\T^\om_{\mu,k}\in \mcS_p(A_\om^2)\Leftrightarrow\T_\mu^\eta\in \mcS_p(A_\eta^2).$$
By Theorem 8 in \cite{PjaRj2016}, we get the desired result. The proof is complete.
\end{proof}

For any $m\in\N$ and $p>0$, the Besov type space $B_{p,m}$ consists of all $f\in H(\D)$ such that
$$\|f\|_{B_{p,m},*}=\left(\int_\D |f^{(m)}(z)|^p(1-|z|)^{mp} \frac{dA(z)}{(1-|z|)^2}\right)^\frac{1}{p}<\infty. $$
Obviously, $\|\cdot\|_{B_{p,m},*}$ is a semi-norm on $B_{p,m}$.
When $mp\leq 1$, $f\in B_{p,m}$ if and only if $f^{(m)}\equiv 0$.
More information about $B_{p,m}$ can be seen in Section 5.3 in \cite{zhu}.

\begin{Corollary}\label{0831-1}
Suppose  $k,n\in\mathbb{Z}$ satisfying $0\leq k<n$ and  $\om\in\hD$.
Let $g\in H(\D)$ such that $T_{g}^{n,k}:A_\om^2\to A_\om^2$ is bounded.
Then,  $T_g^{n,k}\in\mcS_p(A_\om^2)$ if and only if $g\in B_{p,n-k}$.
%$$\int_\D |g^{(n-k)}(z)|^p(1-|z|)^{(n-k)p}\frac{dA(z)}{(1-|z|)^2}<\infty.$$
\end{Corollary}
\begin{proof}
For any $f,h\in A_\om^2$, from Lemma \ref{0413-2}, we have
\begin{align*}
\langle (T_g^{n,k})^*T_g^{n,k}f,h\rangle_{A_\om^2}
&=\langle T_g^{n,k}f,  T_g^{n,k} h\rangle_{A_\om^2}\\
&=2^{2n}\int_\D f^{(k)}(z)\ol{h^{(k)}(z)}|g^{(n-k)}(z)|^2\om^{*n}(z)dA(z).
\end{align*}
Here, $\om^{*n}$ is the $n$-th iteration of $\om^*$, that is, $\om^{*n}=(((\om^*)^*)^{\cdots})^*$.
Obviously, for any given $\varepsilon>0$, $\om^{*n}(z)\approx (1-|z|)^{2n-1}\hat{\om}(z)$ for all $|z|>\varepsilon$.

Replacing   $h$ by $B_\xi^\om$, the reproducing kernel  of $A_\om^2$, then% and $B_\xi^{\om,k}$ be the $k$-th derivative of $B_\xi^\om$, then
$$((T_g^{n,k})^*T_g^{n,k}f)(\xi)=2^{2n}\int_\D f^{(k)}(z)\ol{B_\xi^{\om,k}(z)}|g^{(n-k)}(z)|^2\om^{*n}(z)dA(z).$$
Let $d\mu(z)=|g^{(n-k)}(z)|^2\om^{*n}(z)dA(z)$. From Theorem \ref{0514-1}, we have $(T_g^{n,k})^*T_g^{n,k}\in \mcS_\frac{p}{2}(A_\om^2)$ if and only if
$$\sum\limits_{j=1}^\infty \left(\frac{\mu(\D(a_j,5r))}{(1-|a_j|)^{2k+1}\hat{\om}(a_j)}\right)^\frac{p}{2}<\infty,$$
which equals  to
\begin{align}
\sum_{j=1}^\infty \left(\int_{\D(a_j,5r)}|g^{(m)}(z)|^2(1-|z|)^{2m}\frac{dA(z)}{(1-|z|)^2}\right)^\frac{p}{2}<\infty.\label{0514-2}
\end{align}
Here, $m=n-k$ and $\{a_j\}$ is a $r$-lattice of $\D$ such that $\inf|a_j|>0$.

If (\ref{0514-2}) holds, it is easy to check that
$$ \sum_{j=1}^\infty \left(\int_{\D(a_j,6r)} |g^{(m)}(\xi)|^2(1-|\xi|)^{2m}\frac{dA(\xi)}{(1-|\xi|)^2} \right)^\frac{p}{2}<\infty.$$
By the subharmonicity of $|g^{(m)}|^2$, we have
\begin{align*}
\|g\|_{B_{p,m},*}^p
&\approx \sum_{j=1}^\infty \int_{\D(a_j,5r)} |g^{(m)}(z)|^p(1-|z|)^{mp}\frac{dA(z)}{(1-|z|)^2}  \\
&\lesssim \sum_{j=1}^\infty \int_{\D(a_j,5r)} \left(\int_{D(z,r)}|g^{(m)}(\xi)|^2\frac{dA(\xi)}{(1-|\xi|)^2}\right)^\frac{p}{2}(1-|z|)^{mp}\frac{dA(z)}{(1-|z|)^2}  \\
&\lesssim \sum_{j=1}^\infty \left(\int_{\D(a_j,6r)} |g^{(m)}(\xi)|^2(1-|\xi|)^{2m}\frac{dA(\xi)}{(1-|\xi|)^2} \right)^\frac{p}{2}
     \int_{D(a_j,5r)} \frac{dA(z)}{(1-|z|)^2}  \\
&\approx \sum_{j=1}^\infty \left(\int_{\D(a_j,6r)} |g^{(m)}(\xi)|^2(1-|\xi|)^{2m}\frac{dA(\xi)}{(1-|\xi|)^2} \right)^\frac{p}{2}<\infty.
\end{align*}

Suppose $g\in B_{p,m}$. Let $\xi_j\in\ol{\D(a_j,5r)}$ such that $|g^{(m)}(\xi_j)|=\sup\limits_{z\in \D(a_j,5r)}|g^{(m)}(z)|$.
By the subharmonicity of $|g^{(m)}|^p$, we have
\begin{align*}
\left(\int_{\D(a_j,5r)}|g^{(m)}(z)|^2(1-|z|)^{2m}\frac{dA(z)}{(1-|z|)^2}\right)^\frac{p}{2}
\lesssim&  |g^{(m)}(\xi_j)|^p(1-|\xi_j|^2)^{mp}\\
\lesssim&  \int_{\D(\xi_j,r)}|g^{(m)}(z)|^p(1-|z|^2)^{mp-2}dA(z)\\
\leq & \int_{\D(a_j,6r)}|g^{(m)}(z)|^p(1-|z|^2)^{mp-2}dA(z),
\end{align*}
and
\begin{align*}
\sum_{j=1}^\infty\left(\int_{\D(a_j,5r)}|g^{(m)}(z)|^2(1-|z|)^{2m}\frac{dA(z)}{(1-|z|)^2}\right)^\frac{p}{2}
\lesssim \|g\|_{B_{p,m},*}^p
<\infty.
\end{align*}
The proof is completed.
\end{proof}

\section{More discussions}
Let $H^2$ be the Hardy space, which is also  written as $A_{-1}^2$.
For any $f,g\in H^2$, the inner product of $f$ and $g$ is defined by
$$\langle f,g\rangle_{H^2}=\sum_{k=0}^\infty \hat{f}_k \ol{\hat{g}}_k.$$
So,  $B_z^{-1}(w)=\frac{1}{1-\ol{z}w}$ is the reproducing kernel of $H^2$.
%, that is, for any $f\in H^2$,  $f(z)=\langle f,B_z^{-1}\rangle_{H^2}$.

Then, for any $k=0,1,2,\cdots$ and positive Borel measure $\mu$,  we can define  the generalized Toeplitz operator  $\T_{\mu,k}^{-1}$ on $H^2$, that is, for any $f\in H^2$,
$$\T_{\mu,k}^{-1}f(z)=\int_\D f^{(k)}(w)\ol{D^kB_z^{-1}(w)}d\mu(w).$$
When $k=0$, $\T_{\mu,k}^{-1}$ is the Toeplitz operator $\T_\mu$ on $H^2$  defined in \cite{Ld1987jfa}.
In that paper, Luecking described all these  $\mu$ such that $\T_\mu\in\mathcal{S}(H^2)$.

In  the proof of Theorem \ref{0514-1}, when $k\in\N$,  taking
\begin{align*}
\upsilon(z)=c_1(1-|z|^2)^{k-1}, \,\,\,\eta(z)=c_2(1-|z|^2)^{2k-1}
\end{align*}
for some appropriate positive constant $c_1,c_2$ and substituting  $(1-|z|^2)^{-1}$ for   $\om(z)$,
noting that $A_{-1}^2=H^2$ and $B_z^\om=B_z^{-1}$,
we get the characterizations of $\T_{\mu,k}^{-1} \in \mathcal{S}_p(H^2)$.
That is Theorem \ref{0830-4}.

\begin{theorem}\label{0830-4}
Suppose $k\in\N\cup\{0\}$ and $0<p<\infty$. Let $0<r<\infty$ and $\{a_j\}_{j=1}^\infty$ be a $r$-lattice of \,$\D$.
\begin{enumerate}[(i)]
  \item $\T^{-1}_{\mu,k}$ is bounded on $H^2$ if and only if  $\sup\limits_{z\in\D}\frac{\mu(\D(z,r))}{(1-|z|)^{2k+1}}<\infty$;
  \item  $\T^{-1}_{\mu,k}$ is compact on $H^2$ if and only if   $\lim\limits_{|z|\to 1}\frac{\mu(\D(z,r))}{(1-|z|)^{2k+1}}=0$;
  \item  $\T^{-1}_{\mu,k}\in \mcS_p(H^2)$ if and only if
  $\sum\limits_{j=1}^\infty \left(\frac{\mu(\D(a_j,5r))}{(1-|a_j|)^{2k+1}}\right)^p<\infty.$
\end{enumerate}
\end{theorem}

Let $\tau(z)=\log\frac{1}{|z|}$. Then,
\begin{align*}
\|f\|_{H^2}^2= |f(0)|^2+4\int_\D |f^\p(z)|^2\tau(z)dA(z).
\end{align*}
So, for any $f\in H^2$ and $w\in\D$, from  Lemma \ref{0413-2}, we have
\begin{align*}
(T_g^{n,k})^*T_g^{n,k}f(w)
&=\langle T_g^{n,k}f,  T_g^{n,k} B_w^{-1}\rangle_{H^2}\\
&=2^{2n}\int_\D f^{(k)}(z)\ol{D^kB_w^{-1}(z)}|g^{(n-k)}(z)|^2\tau^{*(n-1)}(z)dA(z).
\end{align*}
Thus, as we got Corollary \ref{0831-1}, we obtain  the following result.

\begin{Corollary}
Suppose  $k,n\in\mathbb{Z}$ satisfying $0\leq k<n$.
Let $g\in H(\D)$ such that $T_{g}^{n,k}:H^2\to H^2$ is bounded.
Then,  $T_g^{n,k}\in\mcS_p(H^2)$ if and only if $g\in B_{p,n-k}$.
%$$\int_\D |g^{(n-k)}(z)|^p(1-|z|)^{(n-k)p}\frac{dA(z)}{(1-|z|)^2}<\infty.$$
\end{Corollary}

\end{document}